\documentclass[12pt, twoside]{article}
\usepackage{amsmath,amsthm,amssymb}
\usepackage{times}
\usepackage{enumerate}
\usepackage{amssymb}
\usepackage{graphicx}

\def\titlerunning#1{\gdef\titrun{#1}}
\makeatletter
\def\author#1{\gdef\autrun{\def\and{\unskip, }#1}\gdef\@author{#1}}

\makeatother

\def\keywords#1{\par\medskip
\noindent\textbf{Keywords.} #1}
\def\subjclass#1{\par\smallskip
\noindent\textbf{MSC (2010):} #1}

\newtheorem{thm}{Theorem}[section]
\newtheorem{cor}[thm]{Corollary}
\newtheorem{lem}[thm]{Lemma}

\theoremstyle{definition}
\newtheorem{defin}[thm]{Definition}
\newtheorem{rem}[thm]{Remark}

\numberwithin{equation}{section}

\newtheorem*{notations}{Notations}

\begin{document}

\baselineskip=17pt

\titlerunning{Radial elliptic problems }

\title{Radial nonlinear elliptic problems with singular or vanishing potentials}

\author{
Marino Badiale\thanks{Partially supported by the PRIN2012 grant ``Variational and perturbative aspects of nonlinear differential problems''},
\ \ Federica Zaccagni }

\date{
\begin{footnotesize}
\emph{
Dipartimento di Matematica ``Giuseppe Peano'' \smallskip\\
Universit\`{a} degli Studi di Torino, Via Carlo Alberto 10, 10123 Torino, Italy \\
e-mail:} marino.badiale@unito.it 
\end{footnotesize}
}
\maketitle

\begin{abstract}
In this paper we prove existence of radial solutions for the nonlinear elliptic problem

\[
-\mathrm{div}(A(|x|)\nabla u)+V(|x|)u=K(|x|)f(u) \quad \text{in }\mathbb{R}^{N}, 
\]

\noindent with suitable hypotheses on the radial potentials $A,V,K$. We first get compact embeddings of radial weighted Sobolev spaces into 
sum of weighted Lebesgue spaces, and then we apply standard variational techniques to get existence results.

\end{abstract}

\section{Introduction}

In this paper we will study the following non linear elliptic equation 
\begin{equation}
-\mathrm{div}(A(|x|)\nabla u)+V(|x|)u=K(|x|)f(u) \quad \text{in }\mathbb{R}^{N},  \label{EQ}
\end{equation}
where $N\geq 3$, $f:\mathbb{R}\rightarrow \mathbb{R}$ is a continuous nonlinearity
satisfying $f\left( 0\right) =0$ and $V\geq 0, A, K>0$ are given radial potentials. 
When $A=1$ the differential operator is the usual laplacian, and this kind of problems 
have been much studied in last years, with different sets of hypotheses on the nonlinearity $f$ and the 
potentials $V,K$. Much work has been devoted in particular to problems in which such potentials can be vanishing or divergent at 
$0$ and $\infty$, because this prevents the use of standard embeddings between Sobolev spaces of radial functions, and new embedding and compactness results must be proved (see for example \cite{Alves-Souto-13}, \cite{Ambr-Fel-Malch}, \cite{BGR}, \cite{BGRnonex}, \cite{BR}, 
\cite{BRpow}, \cite{Be-Gr-Mic}, \cite{Be-Gr-Mic-2}, \cite{BonMerc11}, \cite{Bon-VanSchaft-10}, \cite{SuTian12}, \cite{Su-Wang-Will-2}, \cite{Su-Wang-Will-p}, and the references therein). The case in which the potential $A$ is not trivial has been studied in \cite{Su-Wang-Will-q}, \cite{ChenChen-X}, \cite{Huang-X} for the p-laplacian equation, in \cite{DengHuang1} and \cite{DengHuang2} for bounded domains, and in \cite{ZSW} for exterior domains.
The typical result obtained in these works says, roughly speaking, that given suitable asymptotic behavior at $0$ and $\infty$ for the potentials, there 
is a suitable range of exponent $q$ such that, if $f$ behaves like the power $t^{q-1}$, then problem (\ref{EQ}) has a radial solution. 
\par \noindent In this paper we study problem (\ref{EQ}) using the ideas introduced in \cite{BGRcomp}, \cite{BGRnext}, \cite{BGRlast}. The main novelty of this approach is that the  nonlinearity $f$ is not a pure power as before, but has different power-like behaviors at zero and infinity. The typical example is $f(t)= \min \{t^{q_1-1}, t^{q_2 -1}\}$. Also, we do not introduce hypotheses on the asymptotic behavior of $V,K$, but on their ratio. The typical result is the existence of two intervals $\mathcal{I}_{1} , \mathcal{I}_{2}$ such that if $q_1 \in \mathcal{I}_{1}$ and $q_2 \in \mathcal{I}_{2}$, and $f$ as above, then problem (\ref{EQ}) has a radial solution. When $\mathcal{I}_{1} \cap \mathcal{I}_{2} \not= \emptyset$, it is possible to choose $q_1=q_2 =q \in \mathcal{I}_{1} \cap \mathcal{I}_{2}$, so that $f(t)= t^{q-1}$ and we get results similar to those already known in literature. The main technical device for our results is given by compact embeddings of Sobolev spaces of radial functions in sum of Lebesgue spaces. We refer to \cite{BPR} for an introduction to sum of Lebesgue spaces and to the main results we shall use in this paper.
\par \noindent The paper is organized as follows: after the introduction, in section 2 we introduce the main function spaces we shall use, and prove some preliminary embedding results. In section 3 we introduce some sufficient conditions for compactness of the embeddings, and in section 4 we prove compactness. In section 5 we apply the previous results to get existence and multiplicity results for (\ref{EQ}). Finally, in section 6 we give some concrete examples that, we hope, could help the reader to understand what is new in our results. Notice that the main hypotheses of our results are introduced at the beginning of section 2, while the main results for (\ref{EQ}) are theorems \ref{theorboh3} and \ref{theorboh4}.

\begin{notations}
We end this introductory section by collecting
some notations used in the paper.

\noindent $\bullet $ For every $R>0$, we set $B_{R}:=\left\{ x\in \mathbb{R}%
^{N}:\left| x\right| <R\right\} $.

\noindent $\bullet $ $\omega_N$ is the $(N-1)$-dimensional measure of the surface $\partial B_1 = \left\{ x\in \mathbb{R}%
^{N}:\left| x\right| =1 \right\} $.

\noindent $\bullet $ For any subset $A\subseteq \mathbb{R}^{N}$, we denote $%
A^{c}:=\mathbb{R}^{N}\setminus A$. If $A$ is Lebesgue measurable, $\left|
A\right| $ stands for its measure.

\noindent $\bullet $ By $\rightarrow $ and $\rightharpoonup $ we
respectively mean \emph{strong} and \emph{weak }convergence.

\noindent $\bullet $ $\hookrightarrow $ denotes \emph{continuous} embeddings.

\noindent $\bullet $ $C_{c}^{\infty }(\Omega )$ is the space of the
infinitely differentiable real functions with compact support in the open
set $\Omega \subseteq \mathbb{R}^{N}$; $C_{c,r}^{\infty }(\mathbb{R}%
^{N})$ is the radial subspace of $C_{c}^{\infty }(\mathbb{R}^{N})$. If $B$ is a ball with center 
in $0$, $C_{c,r}^{\infty }(B)$  is the radial subspace of $C_{c}^{\infty }(B )$. 

\noindent $\bullet $ If $1\leq p\leq \infty $ then $L^{p}(A)$ and $L_{%
\mathrm{loc}}^{p}(A)$ are the usual real Lebesgue spaces (for any measurable
set $A\subseteq \mathbb{R}^{N}$). If $\rho :A\rightarrow \left( 0,+\infty
\right) $ is a measurable function, then $L^{p}(A,\rho \left( z\right) dz)$
is the real Lebesgue space with respect to the measure $\rho \left( z\right)
dz$ ($dz$ stands for the Lebesgue measure on $\mathbb{R}^{N}$).

\end{notations}

\bigskip
\bigskip

\section{Hypotheses and pointwise estimates}

Assume $N\geq3$. Let $V$, $K$ and $A$ be three potentials satisfying the following hypothesis:

\begin{itemize}

\item[{[A]}] $A:(0,+\infty)\rightarrow(0,+\infty)$ is a continuous function such that there exist real numbers $2-N<a_{0},a_{\infty}\leq2$ and $c_{0},c_{\infty}>0$ satisfying:
$$
c_{0} \leq \liminf_{r\rightarrow 0^+}\frac{A(r)}{r^{a_{0}}} \leq \limsup_{r\rightarrow 0^+}\frac{A(r)}{r^{a_{0}}} < + \infty 
$$
$$
c_{\infty}\leq \liminf_{r\rightarrow +\infty}\frac{A(r)}{r^{a_{\infty}}} \leq \limsup_{r\rightarrow +\infty}\frac{A(r)}{r^{a_{\infty}}} <+\infty .
$$

\item[{[V]}] $V:(0,+\infty)\rightarrow[0,+\infty)$ belongs to  $ L_{loc}^{1}(0,+\infty).$

\item[{[K]}] $K:(0,+\infty)\rightarrow(0,+\infty)$ belongs to $ L_{loc}^{s}(0,+\infty)$ for some $s>\max\left\{ \frac{2N}{N-a_{0}+2},\;\frac{2N}{N-a_{\infty}+2}\right\} .$

\end{itemize}

\bigskip
\bigskip

\noindent For any $q>1$, we define the weighted Lebesgue space
$L_{K}^{q} = L^{q}(\mathbb{R}^{N}, K(|x|)dx)$ whose norm is $||u||_{L_{K}^{q}}  = \left( \int_{\mathbb{R}^{N}} K(|x|) |u|^{q}dx \right)^{1/q}.$

\bigskip

\begin{defin} For $q_1 , q_2 >1$ we define the sum space ${\cal L}_K = L_{K}^{q_1} + L_{K}^{q_2} $ as

$$ {\cal L}_K = L_{K}^{q_1} + L_{K}^{q_2} = \left\{ u= u_1 + u_2 \,\, \big| \,\, u_i \in L_{K}^{q_i} \right\}  ,$$

\noindent with norm $||u||_{{\cal L}_K }= \inf \left\{ \max\{ ||u_1||_{L_{K}^{q_1}} , ||u_2||_{L_{K}^{q_2}} \}   \, \Big| \, u=u_1 + u_2 , \, u_i \in L_{K}^{q_i}  \right\}$.

\end{defin}

 \noindent By $ L^{q_1} + L^{q_2}$ we mean the sum space obtained when $K \equiv 1$, that is, when the $L_{K}^{q_i}$'s are the usual Lebesgue spaces. We refer to \cite{BPR} for a treatment of such spaces.

\bigskip

\noindent We are now going to prove some pointwise estimates for functions in $C_{c,r}^{\infty }(\mathbb{R}^{N})$, which are the starting point of our arguments. In all this paper, when dealing with a radial function $u$, we will often write, with a little abuse of notation,
$u(x)= u(|x|) =u(r)$ for $|x|=r$.

\bigskip

\begin{rem}
\label{remA}
It is easy to check that the hypothesis $[A]$ implies that, for each $R>0$, there exist $C_{0}=C_{0}(R)>0$ and $C_{\infty}=C_{\infty}(R)>0$ such that 
\begin{equation}
A(|x|)\geq C_{0}|x|^{a_{0}} \quad   \mbox{for all} \,\, 0<|x|\leq R,
\end{equation}
\begin{equation}
A(|x|)\geq C_{\infty}|x|^{a_{\infty}} \quad  \mbox{for all} \,\, |x|\geq R.
\end{equation}
\end{rem}

\bigskip

\begin{lem}
\label{A1}

Assume the hypothesis $[A]$. Fix $R>0$. Then there exists a constant $C=C(N, R, a_{\infty} )>0$ such that, for each $u\in C_{c,r}^{\infty }(\mathbb{R}%
^{N})$, there holds
\begin{equation}
|u(x)|\leq C\, |x|^{-\frac{N+a_{\infty}-2}{2}}\, \left( \int_{B^{c}_{R}} A(|x|) \, |\nabla u |^2 \, dx \right)^{1/2}    \quad for \;R\leq|x|<+\infty.
\end{equation}
\end{lem}

\begin{proof}
If $u\in C_{c,r}^{\infty}(\mathbb{R}^{N})$ and $|x|= r\geq R$, we have
\begin{equation}
-u(r)=\int_{r}^{\infty}u'(s)ds.
\end{equation}
Using the hypothesis $[A]$, we obtain
$$
|u(r)|\leq\int_{r}^{\infty}|u'(s)|ds
$$

$$
=\int_{r}^{\infty}|u'(s)|s^{\frac{N+a_{\infty}-1}{2}}s^{-\frac{N+a_{\infty}-1}{2}}ds
$$

$$
\leq\left(\int_{r}^{\infty}|u'(s)|^{2}s^{N-1}s^{a_{\infty}}ds\right)^{\frac{1}{2}}\left(\int_{r}^{\infty}s^{-(N+a_{\infty}-1)}ds\right)^{\frac{1}{2}}
$$

$$
=  (\omega_{N})^{-\frac{1}{2}}\left(\int_{B_{r}^c}|x|^{a_{\infty}}|\nabla u|^{2}dx\right)^{\frac{1}{2}}\left(\int_{r}^{\infty}s^{-(N+a_{\infty}-1)}ds\right)^{\frac{1}{2}}
$$

$$
\leq \left( C_{\infty} (R) \right)^{1/2} (\omega_{N})^{-\frac{1}{2}}\left(\int_{B_{r}^c} A(|x|)|\nabla u|^{2}dx\right)^{\frac{1}{2}}\left(\int_{r}^{\infty}s^{-(N+a_{\infty}-1)}ds\right)^{\frac{1}{2}}.
$$
As
$$
\int_{r}^{\infty}s^{-(N+a_{\infty}-1)}ds=\frac{r^{-(N+a_{\infty}-2)}}{N+a_{\infty}-2} \, ,
$$
it follows
$$
|u(r)|\leq C\,r^{-\frac{N+a_{\infty}-2}{2}}\, \left( \int_{B^{c}_{R}} A(|x|) \, |\nabla u |^2 \, dx \right)^{1/2} 
$$
\noindent where $C = \left( C_{\infty} (R) \right)^{1/2}  \, (\omega_{N})^{-\frac{1}{2}}\, \left(\frac{1}{N+a_{\infty}-2} \right)^{1/2}= C(N, R, a_{\infty} )$, and this is our thesis.
\end{proof}

\bigskip

\bigskip

\begin{lem}
\label{A2} 

Assume the hypothesis $[A]$. Fix $R >0$. Then there exists a constant $C= C(N, R, a_0 )>0$ such that, for each $u\in   C_{c,r}^{\infty }(B_R)$, there holds

\begin{equation}
|u(x)|\leq C\, |x|^{-\frac{N+a_{0}-2}{2}}\, \left( \int_{B_{R}} A(|x|) \, |\nabla u |^2 \, dx \right)^{1/2} \quad for \;0<|x|<R.
\end{equation}

\end{lem}

\begin{proof}
Let $u\in   C_{c,r}^{\infty }(B_R)$ and take $|x|=r<R$.
Since $u(R)=0$, we have
\begin{equation}
-u(r)=u(R)-u(r)=\int_{r}^{R}u'(s)ds.
\end{equation}
The same arguments of Lemma \ref{A1} yield
$$
|u(r)|\leq\int_{r}^{R}|u'(s)|ds
$$

$$
\leq\left(\int_{r}^{R}|u'(s)|^{2}s^{N-1}s^{a_{0}}ds\right)^{\frac{1}{2}}\left(\int_{r}^{R}s^{-(N+a_{0}-1)}ds\right)^{\frac{1}{2}}
$$

$$
\leq (\omega_{N})^{-\frac{1}{2}} \left(\int_{B_R \backslash B_r }A(|x|)|\nabla u|^{2}dx\right)^{\frac{1}{2}} \left(\int_{r}^{R}s^{-(N+a_{0}-1)}ds\right)^{\frac{1}{2}}
$$

$$
\leq  (\omega_{N})^{-\frac{1}{2}} \left( C_0 (R)\right)^{1/2}\, \left( \int_{B_{R}} A(|x|) \, |\nabla u |^2 \, dx \right)^{1/2} \,  \left(\frac{1}{N+a_{0}-2} \right)^{1/2} r^{-\frac{N+a_{0}-2}{2}}
$$

\noindent that is 

$$
|u(r)|\leq C r^{-\frac{N+a_{0}-2}{2}} \left( \int_{B_{R}} A(|x|) \, |\nabla u |^2 \, dx \right)^{1/2} 
$$

\noindent where $C = (\omega_{N})^{-\frac{1}{2}} \left( C_0 (R )\right)^{1/2}\left(\frac{1}{N+a_{0}-2} \right)^{1/2} = C( N, R, a_0 ) $, which is the thesis.
\end{proof}

\bigskip
\bigskip

\noindent We now introduce another function space.

\bigskip

\begin{defin}
$$S_A = \left\{ u \in   C_{c,r}^{\infty }(\mathbb{R}
^{N}) \, \Big| \,  \int_{\mathbb{R}
^{N}} A(|x|) \, |\nabla u |^2 \, dx  < + \infty \right\}$$
 
\noindent $S_A$ is a subspace of $ C_{c,r}^{\infty }(\mathbb{R}
^{N})$. We define on $S_A$ the norm $||u||_A =  \left( \int_{\mathbb{R}
^{N}} A(|x|) \, |\nabla u |^2 \, dx \right)^{1/2}$. 
\end{defin}

\bigskip

\begin{defin}
Let $2-N<a_{0},\, a_{\infty}\leq2$. We define the following real numbers
\begin{equation}
p_{0}:=\frac{2N}{N+a_{0}-2},\quad p_{\infty}:=\frac{2N}{N+a_{\infty}-2}.
\end{equation}
\end{defin}

\begin{rem}
Notice that $p_{0}, p_{\infty}\geq 2$.
\end{rem}

\bigskip

\noindent The main property of $S_A$ is given by the following lemma.

\bigskip

\begin{lem}
\label{A4}

Consider $A$ satisfying the hypothesis $[A]$. The embedding
\begin{equation*}
S_A \hookrightarrow L^{p_{0}}(\mathbb{R}^{N})+L^{p_{\infty}}(\mathbb{R}^{N})
\end{equation*}
is continuous.
 \end{lem}

\begin{proof}
Let us fix $ R >0$ and $C>0$ such that $A(|x|)\geq C |x|^{a_{\infty}}$ for $|x|\geq R$. Take $u \in S_A$. We want to estimate $\int_{B_{R}^{c}}|u|^{p_{\infty}}dx$. With an integration by parts, and using Lemma \ref{A1}, we obtain
$$
\int_{B_{R}^{c}}|u|^{p_{\infty}}dx=\omega_{N}\int_{R}^{\infty}r^{N-1}|u(r)|^{p_{\infty}}dr
$$

$$
\leq\frac{2\omega_{N}}{N+a_{\infty}-2}\int_{R}^{\infty}r^{N}|u(r)|^{p_{\infty}-1}|u'(r)|dr
$$

$$
\leq\frac{2\omega_{N}}{N+a_{\infty}-2}\left(\int_{R}^{\infty}r^{a_{\infty}}|u'(r)|^{2}r^{N-1}dr\right)^{\frac{1}{2}}\left(\int_{R}^{\infty}r^{(N-\frac{N-1+a_{\infty}}{2})2}|u(r)|^{2(p_{\infty}-1)}dr\right)^{\frac{1}{2}}
$$

$$
\leq C\frac{2\omega_{N}^{\frac{1}{2}}}{N+a_{\infty}-2}\left(\int_{B_{R}^{c}}A(|x|)|\nabla u|^{2}dx\right)^{\frac{1}{2}}\left(\int_{R}^{\infty}r^{N-1}|u(r)|^{p_{\infty}}r^{2-a_{\infty}}|u(r)|^{p_{\infty}-2}dr\right)^{\frac{1}{2}}
$$

$$
\leq C^{\frac{p_{\infty}-2}{2}}\frac{2\omega_{N}^{\frac{1}{2}}}{N+a_{\infty}-2}\left(\int_{B_{R}^{c}}A(|x|)|\nabla u|^{2}dx\right)^{\frac{p_{\infty}}{4}}\left(\int_{R}^{\infty}r^{N-1}|u(r)|^{p_{\infty}}dr\right)^{\frac{1}{2}}
$$

$$
\leq C \left(\int_{B_{R}^{c}}A(|x|)|\nabla u|^{2}dx\right)^{\frac{p_{\infty}}{4}}  \left(\int_{B_{R}^{c}}|u|^{p_{\infty}}dx\right)^{\frac{1}{2}},
$$

\noindent where $C= C(N, R, a_{\infty })$ may change from line to line. From this we obtain
\begin{equation}
\left(\int_{B_{R}^{c}}|u|^{p_{\infty}}dx\right)^{\frac{1}{p_{\infty}}}\leq C \, \left(\int_{B_{R}^{c}}A(|x|)|\nabla u|^{2}dx\right)^{\frac{1}{2}}  \leq C\|u\|_A.\label{eq:2222}
\end{equation}

\noindent Assume now that $C= C_0 (R+1)>0$ is such that $A(|x|)\geq C |x|^{a_{0}}$ for $0<|x|=r \leq R+1$. We want to estimate the integral $\int_{B_{R}}|u|^{p_{0}}dx$.
Let us define a radial cut-off function $\rho(x)\in C_{0,r}^{\infty}(\mathbb{R}^{N})$ such that $\rho (x) \in [0,1]$ for all $x$ and
$$                                                   
\rho(x)= \rho (|x|)=\begin{cases}
1 & 0<|x|\leq R\\
0 & |x|\geq R+\frac{1}{2}
\end{cases}
$$
\noindent Of course $\rho u\in C_{c,r}^{\infty }(B_{R+1}) $ and we can employ Lemma \ref{A2}. With the same computations used in the previous case, we have 
$$
\int_{B_{R+1}}|\rho u|^{p_{0}}dx
$$

$$
\leq C \, \left(\int_{B_{R+1}}A(|x|)|\nabla(\rho u)|^{2}dx\right)^{\frac{p_{0}}{4}}\left(\int_{B_{R+1}}|\rho u|^{p_{0}}dx\right)^{\frac{1}{2}},
$$

\noindent where $C=C(N, R, a_0 )$. From this we derive
$$
\int_{B_{R+1}}|\rho u|^{p_{0}}dx\leq C\left(\int_{B_{R+1}}A(|x|)|\nabla(\rho u)|^{2}dx\right)^{\frac{p_{0}}{2}}.
$$

\noindent Using the continuity of $A$ in the compact set $\overline{ B_{R+1}\setminus B_{R} }$ and thanks to Lemma \ref{A1}, we obtain

$$
\int_{B_{R+1}}A(|x|)|\nabla(\rho u)|^{2}dx
$$

$$
\leq C\int_{B_{R+1}}A(|x|)(|\nabla u|^{2}\rho^{2}+|\nabla\rho|^{2}u^{2})dx
$$

$$
\leq C\left(\int_{B_{R+1}}A(|x|)|\nabla u|^{2}dx+\int_{B_{R+1}\setminus B_{R}}|\nabla\rho|^{2}u^{2}dx\right)
$$

$$
\leq C\left(\|u\|_{A}^{2}+\|u\|_{A}^{2}\int_{B_{R+1}\setminus B_{R}}|x|^{-(N+a_{\infty}-2)}dx\right)
$$

$$
\leq C\|u\|_{A}^{2},
$$

\noindent where $C=C(N, R, a_{\infty})$. Thus, it holds
$$
\int_{B_{R+1}}|\rho u|^{p_{0}}dx\leq C\|u\|_{A}^{p_{0}}, 
$$

\noindent where $C=C(N, R, a_{0}, a_{\infty})$. As a consequence we get
$$
\int_{B_{R}}|u|^{p_0}dx=\int_{B_{R}}|\rho u|^{p_{0}}dx\leq\int_{B_{R+1}}|\rho u|^{p_{0}}dx\leq C\|u\|_{A}^{p_{0}},
$$

\noindent hence

\begin{equation}
\left(\int_{B_{R}}|u|^{p_{0}}dx\right)^{\frac{1}{p_{0}}}\leq C\|u\|_A.\label{eq:333-1}
\end{equation}

\noindent Now we define 
$$
\begin{cases}
\bar{u}_{1}:=u\chi_{B_{R}}\\
\bar{u}_{2}:=u\chi_{B_{R}^{c}}
\end{cases},
$$

\noindent so we obtain $$u=\bar{u}_{1}+\bar{u}_{2}$$ with $\bar{u}_{1}\in L^{p_{0}}(\mathbb{R}^{N})$
and $\bar{u}_{2}\in L^{p_{\infty}}(\mathbb{R}^{N})$, from which it follows

$$
\|u\|_{L^{p_{0}}+L^{p_{\infty}}
   }\leq\|\bar{u}_{1}\|_{L^{p_{0}}(\mathbb{R}^{N})}+\|\bar{u}_{2}\|_{L^{p_{\infty}}(\mathbb{R}^{N})}
$$

$$
=\|u\|_{L^{p_{0}}(B_{R})}+\|u\|_{L^{p_{\infty}}(B_{R}^{c})}\leq C\|u\|_A.
$$

\noindent This hold for any $u\in S_A$, with a constant $C=C(N, R, a_0 , a_{\infty})$. Thus, the embedding
$$
S_A \hookrightarrow L^{p_{0}}(\mathbb{R}^{N})+L^{p_{\infty}}(\mathbb{R}^{N})
$$

\noindent is continuous.
\end{proof}

\bigskip

\noindent We want now to introduce the completion of $S_A$ with respect to $||\cdot ||_A$. 

\bigskip

\begin{defin}
$D_{A}$ is the space of all $u \in L^{p_0} + L^{p_{\infty}}$ for which there is a sequence $\{u_n \}_n \subset S_A$ such that

\begin{itemize}  
\item $u_n \rightarrow u$ in $L^{p_0} + L^{p_{\infty}}$.

\item $\{u_n \}_n $ is a Cauchy sequence with respect to $|| \cdot ||_A$.

\end{itemize}

\end{defin}

\bigskip
\bigskip

 \noindent Of course, $D_A$ is a linear subspace of $L^{p_0} + L^{p_{\infty}}$. From the previous results we deduce the following two lemmas, which say that $D_A $ 
is the completion of $S_A$ wth respect to $|| \cdot ||_A$. The arguments are essentially standard, so we will skip the details.

\bigskip

\begin{lem}
 
Assume $[A]$. Let $u \in D_A$. Then $u$ has weak derivatives $D_i u$ in the open set $\Omega=\mathbb{R}^{N} \backslash \{0 \}$ ($i=1,..., N$) and it holds $D_i u \in L^2_{loc}(\Omega )$. 
\par \noindent If $\{u_n \}_n $ is the sequence in $S_A$ given by the definition of $u \in D_A$, then  
$$\int_{\mathbb{R}^{N}} A(|x|) |\nabla u(x) - \nabla u_n|^2 \, dx \rightarrow 0.$$

\noindent In particular $\int_{\mathbb{R}^{N}} A(|x|) |\nabla u(x)| \, dx < + \infty $, and $||u||_A =\left( \int_{\mathbb{R}^{N}} A(|x|) |\nabla u(x)| \, dx
	\right)^{1/2}$ is a norm on $D_A$.
	\end{lem}
		\begin{proof}
	The proof is a simple exercise on weak derivatives, and we leave it to the reader.
	\end{proof}
	
	\bigskip
	
	\begin{lem}
	Assume $[A]$. If we consider the space $D_A$ endowed with the norm $|| \cdot ||_A $ we have
	$$D_{A}\hookrightarrow L^{p_{0}}(\mathbb{R}^{N})+L^{p_{\infty}}(\mathbb{R}^{N})$$
	
	\noindent with continuous embedding.
	\par \noindent Furthermore, $D_A$ endowed with the norm $|| \cdot ||_A $ is complete, so it is an Hilbert space.
	\end{lem}
	\begin{proof}
To prove continuity of embedding, let $||u||_{\cal L}$ be the norm of $u\in L^{p_{0}}+L^{p_{\infty}}$. For $u \in D_A$ let $\{u_n \}_n $ be the sequence in $S_A$ given by the definition. By the previous lemma we have $||u_n ||_{\cal L} \leq C \, ||u_n ||_{D_A}$, for a suitable positive constant $C$. But $||u_n -u||_{\cal L} \rightarrow 0$ by hypothesis and $||u_n -u||_{D_A} \rightarrow 0$ by $(i)$, so we get 
	$$||u ||_{\cal L} \leq C \, ||u ||_{D_A},$$
	
	\noindent which is the thesis.
	\par \noindent To prove completeness, let $\{v_n \}_n \subset D_A$ be a Cauchy sequence. From $(ii)$, it is a Cauchy sequence in $L^{p_{0}}+L^{p_{\infty}}$, so 
	$ u_n \rightarrow u$ in $L^{p_{0}}+L^{p_{\infty}}$. By $(i)$, for each $n$ we can choose $v_n \in S_A$ such that $||u_n - v_n ||_A \leq 1/n$. Then we have
	$$||u-v_n ||_{\cal L}\leq ||u-u_n ||_{\cal L} +||u_n-v_n ||_{\cal L}\leq ||u-u_n ||_{\cal L} + C||u_n-v_n ||_{D_A}\rightarrow 0 .$$
	
	\noindent Also, it is 
	
	$$||v_m-v_n ||_{D_A} \leq ||v_m-u_m ||_{D_A} + ||u_n-u_m ||_{D_A} +||u_n-v_n ||_{D_A}\leq \frac{1}{m} + ||u_n-u_m ||_{D_A} + \frac{1}{n},$$
	
	\noindent so also $\{v_n \}_n$ is a Cauchy sequence with respect to $||\cdot||_{D_A}$. By definition of $D_A $ we have $u \in D_A$, and by $(i)$ $||u-v_n ||_A \rightarrow 0$. Now we get
	
	$$||u- u_n ||_A \leq ||u- v_n ||_A + ||v_n- u_n ||_A \rightarrow 0 ,$$
	
	\noindent and the proof is complete. 
	
\end{proof}

\bigskip

\noindent We will also need the following corollary.

\begin{cor}
\label{C1} Consider $A$ satisfying the hypothesis $[A]$. We define 

$$
a:=\mathrm{max}\{a_{0},a_{\infty}\},\quad p_{*}:=\frac{2N}{N+a-2}=\mathrm{min}\{ p_{0}, p_{\infty} \} \geq 2.
$$

Then, for any finite measure set $E\subseteq\mathbb{R}^{N}$, the embedding

$$
D_{A}\hookrightarrow L^{p_{*}}(E)
$$

is continuous.
\end{cor}
\begin{proof}
This derives from the previous lemma together with prop. 2.17 $ii)$ of \cite{BPR}. \end{proof}

\bigskip

\noindent We now define another function space.

\bigskip

\begin{defin}
$$ X = \left\{ u \in D_{A} \, : \, \int_{\mathbb{R}^{N}} V(|x|) \, | u |^2 \, dx < +\infty \right\}.$$

\noindent with norm

$$ ||u||^2 = ||u||_{X}^2 = \int_{\mathbb{R}^{N}} A(|x|) \, |\nabla u |^2 + \int_{\mathbb{R}^{N}} V(|x|) \, | u |^2 \, dx .$$

\end{defin}

\bigskip
\bigskip

\noindent $X$ is an Hilbert spaces with respect to the norm $ ||\cdot ||_X$. We will write $(u|v) $ for the scalar product in $X$, that is

$$(u|v) = \int_{\mathbb{R}^{N}} A(|x|) \, \nabla u \cdot \nabla v \, dx + \int_{\mathbb{R}^{N}} V(|x|) \, u \, v \, dx .$$

\par \noindent We look for weak solutions of equation (\ref{EQ}) in the space $X$. This means that a solution (\ref{EQ}) is a function $u \in X$ such that, for all $h \in X$, it holds 

$$\int_{\mathbb{R}^{N}} A(|x|) \nabla u \nabla h \, dx + \int_{\mathbb{R}^{N}} V(|x|) u \, h dx -  \int_{\mathbb{R}^{N}} K(|x|) \, f(u) h \,dx =0 .$$

\noindent We will obtain such weak solutions by standard variational methods, that is we will introduce (in section 5) a functional on $X$ whose critical points are weak solutions. To get such critical points we need, as usual, some compactness properties for the functional, which we will derive from compactness of suitable embeddings. So in the following sections we will prove that the space $X$ is compactly embedded in $L_{K}^{q_1} + L_{K}^{q_2}$, for suitable $q_1 , q_2$. 
The following lemma is a step to obtain these compact embeddings.

\bigskip

\begin{lem}
\label{A5}

Assume $[A]$, $[V]$, $[K]$. Let $a=\mathrm{max}\{a_{0},a_{\infty}\}$ and take $1<q<\infty$, $0<r<R$. Then there exists a constant $C=C(N,a_{0}, a_{\infty},r,R,q,s)>0$ such that, for any $u , h\in X $ we have
$$
\frac{\int_{B_{R}\setminus B_{r}}K(|x|)|u|^{q-1}|h|dx}{C\|K(|\cdot|)\|_{L^{s}(B_{R}\setminus B_{r})}}\leq\begin{cases}
\left(\int_{B_{R}\setminus B_{r}}|u|^{2}dx\right)^{\frac{q-1}{2}}\|h\| & if\:q\leq\tilde{q}\\
\left(\int_{B_{R}\setminus B_{r}}|u|^{2}dx\right)^{\frac{\tilde{q}-1}{2}}\|u\|^{q-\tilde{q}}\|h\| & if\:q>\tilde{q}
\end{cases}
$$

where $\tilde{q}=2\left(1+\frac{1}{N}-\frac{1}{s}\right)-\frac{a}{N}$. \end{lem}

\begin{proof}
We denote with $\sigma$ the conjugate exponent of $p_{*}$, i.e. 
$\sigma=\frac{2N}{N-a+2} = \max \left \{ \frac{2N}{N-a_{0}+2}, \frac{2N}{N-a_{\infty}+2}  \right \}$. Thanks to H\"older's inequality, initially applied with $p_{*}>1$ and then with $\frac{s}{\sigma}>1$, we obtain
$$
\int_{B_{R}\setminus B_{r}}K(|x|)|u|^{q-1}|h|dx
$$

$$
\leq\left(\int_{B_{R}\setminus B_{r}}K(|x|)^{\sigma}|u|^{(q-1)\sigma}dx\right)^{\frac{1}{\sigma}}\left(\int_{B_{R}\setminus B_{r}}|h|^{p_{*}}dx\right)^{\frac{1}{p_{*}}}
$$

$$
\leq\left(\left(\int_{B_{R}\setminus B_{r}}K(|x|)^{s}dx\right)^{\frac{\sigma}{s}}\left(\int_{B_{R}\setminus B_{r}}|u|^{(q-1)\sigma(\frac{s}{\sigma})'}dx\right)^{\frac{1}{(\frac{s}{\sigma})'}}\right)^{\frac{1}{\sigma}} C \|h\|
$$

$$
\leq C \|K(|\cdot|)\|_{L^{s}(B_{R}\setminus B_{r})}\|h\| \left(\int_{B_{R}\setminus B_{r}}|u|^{2\frac{q-1}{\tilde{q}-1}}dx\right)^{\frac{\tilde{q}-1}{2}},
$$
\noindent where we use the following computations
$$
\frac{1}{\left(\frac{s}{\sigma}\right)^{'}}=1-\frac{\sigma}{s}=\frac{s(N-a+2)-2N}{s(N-a+2)}$$

\noindent hence
$$
\left(\frac{s}{\sigma}\right)^{'}=\frac{s(N-a+2)}{s(N-a+2)-2N}
$$
\noindent so that

$$
\sigma\left(\frac{s}{\sigma}\right)^{'}=\frac{2N}{N-a+2}\frac{s(N-a+2)}{s(N-a+2)-2N}=2\frac{sN}{sN+2s-2N-as}=\frac{2}{\tilde{q}-1}.
$$

\begin{itemize}

\item If $q=\tilde{q}$, the proof is over. 
\item If $q<\tilde{q}$, we apply H\"older's inequality again, with conjugate exponents $\frac{\tilde{q}-1}{q-1}>1$, $\frac{\tilde{q}-1}{\tilde{q}-q}$, obtaining
$$
\int_{B_{R}\setminus B_{r}}K(|x|)|u|^{q-1}|h|dx
$$

$$
\leq C \, \|K(|\cdot|)\|_{L^{s}(B_{R}\setminus B_{r})}\|h\|\left(|B_{R}\setminus B_{r}|^{\frac{\tilde{q}-q}{\tilde{q}-1}}\left(\int_{B_{R}\setminus B_{r}}|u|^{2}dx\right)^{\frac{q-1}{\tilde{q}-1}}\right)^{\frac{\tilde{q}-1}{2}}
$$

$$
=C\, \|K(|\cdot|)\|_{L_{s}(B_{R}\setminus B_{r})}\|h\|   \left(\int_{B_{R}\setminus B_{r}}|u|^{2}dx\right)^{\frac{q-1}{2}}.
$$
\item If $q>\tilde{q}$, we have $\frac{q-1}{\tilde{q}-1}>1$.
Thanks to Lemma  \ref{A1} we obtain

$$
\int_{B_{R}\setminus B_{r}}K(|x|)|u|^{q-1}|h|dx
$$

$$
\leq C \|K(|\cdot|)\|_{L^{s}(B_{R}\setminus B_{r})}\|h\|\left(\int_{B_{R}\setminus B_{r}}|u|^{2\frac{q-1}{\tilde{q}-1}-2}|u|^{2}dx\right)
$$

$$
\leq  C\|K(|\cdot|)\|_{L^{s}(B_{R}\setminus B_{r})}\|h\|\left(\left(\frac{C\|u\|}{r^{\frac{N+a_{\infty}-2}{2}}}\right)^{2\frac{q-\tilde{q}}{\tilde{q}-1}}\int_{B_{R}\setminus B_{r}}|u|^{2}dx\right)^{\frac{\tilde{q}-1}{2}}
$$

$$
\leq C \|K(|\cdot|)\|_{L^{s}(B_{R}\setminus B_{r})}\|h\|   \|u\|^{q-\tilde{q}}   \left(\int_{B_{R}\setminus B_{r}}|u|^{2}dx\right)^{\frac{\tilde{q}-1}{2}}.
$$
\end{itemize} In all the previous computations, $C$ may mean different positive constants, depending only on $N,a_{0}, a_{\infty}, r, R, q,s$.
\end{proof}

\bigskip

\begin{lem}
\label{A6}

Consider $A$ satisfying the hypothesis $[A]$. Let $\Omega $ be a smooth bounded open set such that $\overline{\Omega}\subset \mathbb{R}^{N}\setminus\{0\}$. Then the embedding

$$
D_{A} \hookrightarrow L^{2}(\Omega)
$$

\noindent is continuous and compact.
\end{lem}

\begin{proof}
The continuity of the embedding is obvious thanks to Corollary \ref{C1} and the fact that $p_* \geq 2$. 
We prove now the compactness of the embedding. Let $\{u_{n}\}_n$ be a bounded sequence in
$D_{A} $. By continuity of the embedding we obtain

$$
\|u_n \|_{L^{2}(\Omega)}\leq C.
$$

\noindent Moreover as the function $A(x)$ is continuous and strictly positive in the compact set $\overline {\Omega }$, there holds

$$
\int_{\Omega} |\nabla u_n |^2 dx \leq C \int_{\Omega} A(|x|) |\nabla u_n |^2 dx
\leq C ||u_n ||_{D_{A}}^2\leq C.
$$

\noindent Thus, $\{u_{n}\}_n$ is bounded also in the space $H^{1}(\Omega)$. Thanks to Rellich's Theorem, $\{u_{n}\}_n$ has a convergent subsequence in $L^{2}(\Omega)$, and this gives our thesis.
\end{proof}

\bigskip
\bigskip

\section{Some results on embeddings}

Following \cite{BGRcomp} we now introduce some new functions, whose study will help us in getting conditions for compactness.

\begin{defin}
For $q>1$ and $R>0$ define

$$
\mathcal{S}_{0}(q,R):=\sup_{u\in X,\|u\| =1}\int_{B_{R}}K(|x|)|u|^{q}dx
$$

$$
\mathcal{S}_{\infty}(q,R):=\sup_{u\in X,\|u\| =1}\int_{B_{R}^{c}}K(|x|)|u|^{q}dx.
$$

\end{defin}

\bigskip

\begin{thm}

Assume $N\geq3$ and $[A]$, $[V],$ $[K]$. Take $q_{1},q_{2}>1$.
\begin{enumerate}

\item If
\begin{equation}
\mathcal{S}_{0}(q_{1},R_{1})<\infty\;\;and \;\;\mathcal{S}_{\infty}(q_{2},R_{2})<\infty\;\;for\; some \;R_{1},R_{2}>0,\label{eq:27}
\end{equation}
then 
$$X \hookrightarrow L_{K}^{q_{1}}(\mathbb{R^{\mathit{N}}})+L_{K}^{q_{2}}(\mathbb{R^{\mathit{N}}}), $$

\noindent with continuous embedding.

\item If
\begin{equation}
\lim_{R\rightarrow 0^{+}}\mathcal{S}_{0}(q_{1},R)=\lim_{R\rightarrow+\infty}\mathcal{S}_{\infty}(q_{2},R)=0,\label{eq:28}
\end{equation}
then the embedding of $X$ into $L_{K}^{q_{1}}(\mathbb{R^{\mathit{N}}})+L_{K}^{q_{2}}(\mathbb{R^{\mathit{N}}})$ is compact.

\end{enumerate}
\end{thm}

\begin{proof}

As to $1$, we remark that $\mathcal{S}_{0}$ and $\mathcal{S}_{\infty}$ are monotone, so it is not restrictive to assume $R_{1}<R_{2}$.
Assume $u\in X,\;u\neq0$, then
\begin{equation}
\int_{B_{R_{1}}}K(|x|)|u|^{q_{1}}dx=\|u\|^{q_{1}}\int_{B_{R_{1}}}K(|x|)\frac{|u|^{q_{1}}}{\|u\|^{q_{1}}}dx\leq\|u\|^{q_{1}}\mathcal{S}_{0}(q_{1},R_{1})\label{eq:8}
\end{equation}
and, in the same way,

\begin{equation}
\int_{B_{R_{2}}^{c}}K(|x|)|u|^{q_{2}}dx\leq\|u\|^{q_{2}}\mathcal{S}_{\infty}(q_{2},R_{2}).\label{eq:9}
\end{equation}

\noindent Here $||u|| = ||u||_X $ is the norm in $X$.

\noindent Using lemma \ref{A5} (with $h=u$) and lemma \ref{A6} we obtain, for a suitable $C>0$,
independent from $u$, 
\begin{equation}
\int_{B_{R_{2}}\setminus B_{R_{1}}}K(|x|)|u|^{q_{1}}dx\leq C\|u\|^{q_{1}}.\label{eq:10}
\end{equation}

\noindent Hence $u\in L_{K}^{q_{1}}(B_{R_{2}})\cap L_{K}^{q_{2}}(B_{R_{2}}^{c})$
so, by proposition 2.3 in \cite{BPR}, $u\in L_{K}^{q_{1}}+L_{K}^{q_{2}}.$
Moreover, if $u_{n}\rightarrow 0$ in $X$, then, thanks to (\ref{eq:8}), (\ref{eq:9}) and (\ref{eq:10}), we obtain

$$
\int_{B_{R_{2}}}K(|x|)|u_{n}|^{q_{1}}dx+\int_{B_{R_{2}}^{c}}K(|x|)|u_{n}|^{q_{2}}dx \rightarrow 0 \quad \mbox{as} \, \, n\rightarrow\infty .
$$

\noindent It follows that $u_{n}\rightarrow 0$ in $L_{K}^{q_{1}}+L_{K}^{q_{2}}$ (see proposition 2.7 in \cite{BPR}).
\par \noindent As to $2$, we assume hypothesis (\ref{eq:28}) and let $u_{n}\rightharpoonup 0$ in $X$. Then, $\{u_{n}\}$ is bounded in $X$. Thanks to (\ref{eq:8}) and (\ref{eq:9}) we get that, for a fixed 
$\varepsilon>0$, it is possible to obtain $R_{\epsilon}$ and $r_{\epsilon}$ such that $R_{\epsilon}>r_{\epsilon}>0$ and for all $n\in\mathbb{N},$
$$
\int_{B_{r_{\epsilon}}}K(|x|)|u_{n}|^{q_{1}}dx\leq\|u_{n}\|^{q_{1}}\mathcal{S}_{0}(q_{1},r_{\epsilon})\leq\sup_{n}\|u_{n}\|^{q_{1}}\, \mathcal{S}_{0}(q_{1},r_{\epsilon})<\frac{\epsilon}{3}
$$
\noindent and
$$
\int_{B_{R\epsilon}^{c}}K(|x|)|u_{n}|^{q_{2}}dx\leq\sup_{n}\|u_{n}\|^{q_{2}}\, \mathcal{S}_{\infty}(q_{1},R_{\epsilon})<\frac{\epsilon}{3}.
$$

\noindent Thanks to Lemma \ref{A5} and to the boundedness of $\{u_{n}\}_n$ in $X$,
there exist two constants $C, \mu>0$, independent from $n$, such tat
$$
\int_{B_{R\epsilon}\setminus B_{r_{\epsilon}}}K(|x|)|u_{n}|^{q_{2}}dx\leq C\left(\int_{B_{R\epsilon}\setminus B_{r_{\epsilon}}}|u_{n}|^{2}dx\right)^{\mu}\rightarrow 0 \quad \mbox{as} \, \, n\rightarrow\infty ,
$$

\noindent thanks to Lemma \ref{A6}. For $n$ big enough we have 
$$
\int_{B_{R\epsilon}}K(|x|)|u_{n}|^{q_{1}}dx+\int_{B_{R\epsilon}^{c}}K(|x|)|u_{n}|^{q_{2}}dx<\epsilon .
$$

\noindent  From this, thanks to proposition 2.7 in \cite{BPR}, we get $u_{n}\rightarrow 0$ in $L_{K}^{q_{1}}+L_{K}^{q_{2}}$.
\end{proof}

\bigskip
\bigskip

\section{Compactness of embeddings}

We start this section proving the following two lemmas, which give the most important technical steps for our compactness results. For future purposes, these two lemmas are stated in a form which is a little more general than needed in the present paper.

\smallskip{}

\begin{lem}
\label{A7} Let $N\geq3$ and $R>0$. Assume $[A]$, $[V],$ $[K]$. If we assume
$$
\Lambda:=\mathsf{\mathnormal{\mathbf{\mathrm{ess\!\sup_{\mathit{x\in B_{R}}}}}}}\frac{K(|x|)}{|x|^{\alpha}V(|x|)^{\beta}}<+\infty\quad for\;some\;0\leq\beta\leq1\;and\;\alpha\in\mathbb{R},
$$

\noindent then, $\forall u,h \in X$ and $\forall q>\mathrm{max}\{1,2\beta\}$ we have
$$
\int_{B_{R}}K(|x|)|u|^{q-1}|h|dx
$$

$$
\leq\begin{cases}
\Lambda ||u||^{q-1}\, C \, \left(\int_{B_{R}}|x|^{\frac{\alpha-\nu(q-1)}{N-a_{0}+2(1-2\beta+a_{0}\beta)}2N}dx\right)^{\frac{N-a_{0}+2(1-2\beta+a_{0}\beta)}{2N}}\|h\| & 0\leq\beta\leq\frac{1}{2}\\
\Lambda ||u||^{q-1}  \, C \, \left(\int_{B_{R}}|x|^{\frac{\alpha-\nu(q-2\beta)}{1-\beta}}dx\right)^{1-\beta} \|h\| & \frac{1}{2}<\beta<1\\
\Lambda ||u||^{q-2}  \, C \, \left(\int_{B_{R}}|x|^{2\alpha-2\nu(q-2)}V(|x|)|u|^{2}dx\right)^{\frac{1}{2}}\|h\| & \beta=1
\end{cases}
$$

\noindent where $\nu:=\frac{N+a_{0}-2}{2}$ and $C=C (N,R, a_0, a_{\infty})$. \end{lem}

\begin{proof}

We study the various cases separately:

\begin{enumerate}
\item

If $\beta=0$, we apply H\"older's inequality with conjugate exponents $p_{0}={\frac{2N}{N+a_{0}-2}}$, ${\frac{2N}{N-a_{0}+2}}$. We apply also 
(\ref{eq:333-1}) and Lemma \ref{A2} and we get

$$
\frac{1}{\Lambda}\int_{B_{R}}K(|x|)|u|^{q-1}|h|dx\leq\int_{B_{R}}|x|^{\alpha}|u|^{q-1}|h|dx
$$

$$
\leq\left(\int_{B_{R}}\left(|x|^{\alpha}|u|^{q-1}\right)^{\frac{2N}{N-a_{0}+2}}dx
\right)^{\frac{N-a_{0}+2}{2N}}   \left(\int_{B_{R}}|h|^{p_{0}}dx\right)^{\frac{1}{p_{0}}}
$$

$$
\leq ||u||^{q-1} \, C \, \left(\int_{B_{R}}|x|^{\frac{\alpha-\nu(q-1)}{N-a_{0}+2}2N}dx\right)^{\frac{N-a_{0}+2}{2N}}\|h\|.
$$

\item If $0<\beta<1/2$ then it is possible to apply H\"older's inequality with conjugate exponents $\frac{1}{\beta}$, $\frac{1}{1-\beta}$:
$$
\frac{1}{\Lambda}\int_{B_{R}}K(|x|)|u|^{q-1}|h|dx
$$

$$
\leq\int_{B_{R}}|x|^{\alpha}V(|x|)^{\beta}|u|^{q-1}|h|dx=\int_{B_{R}}|x|^{\alpha}|u|^{q-1}|h|^{1-2\beta}V(|x|)^{\beta}|h|^{2\beta}dx
$$

$$
\leq\left(\int_{B_{R}}\left(|x|^{\alpha}|u|^{q-1}|h|^{1-2\beta}\right)^{\frac{1}{1-\beta}}dx\right)^{1-\beta}\left(\int_{B_{R}}V(|x|)|h|^{2}dx\right)^{\beta}
$$

$$
\leq\left(\int_{B_{R}}\left(|x|^{\alpha}|u|^{q-1}|h|^{1-2\beta}\right)^{\frac{1}{1-\beta}}dx\right)^{1-\beta}\|h\|^{2\beta}.
$$

\noindent We apply again H\"older's inequality with exponent $\frac{1-\beta}{1-2\beta}\, p_{0} > 1$. Its conjugate exponent is given through the formula

$$
\frac{1}{\left(\frac{1-\beta}{1-2\beta}p_{0}\right)^{'}}=1-\frac{1-2\beta}{p_{0}(1-\beta)}=\frac{p_{0}(1-\beta)-(1-2\beta)}{p_{0}(1-\beta)}
$$

$$
=\frac{\frac{2N}{N+a_{0}-2}(1-\beta)-(1-2\beta)}{\frac{2N}{N+a_{0}-2}(1-\beta)}=\frac{2N(1-\beta)-(N+a_{0}-2)(1-2\beta)}{2N(1-\beta)}
$$

$$
=\frac{N-a_{0}+2(1-2\beta+a_{0}\beta)}{2N(1-\beta)}.
$$

\noindent We obtain that
$$
\frac{1}{\Lambda}\int_{B_{R}} K(|x|)|u|^{q-1}|h|dx
$$

$$
\leq\left(\left(\int_{B_{R}} \left(|x|^{\frac{\alpha}{1-\beta}}|u|^{\frac{q-1}{1-\beta}}\right)^{\left(\frac{1-\beta}{1-2\beta}p_{0}\right)^{'}}dx\right)^{\frac{1}{\left(\frac{1-\beta}{1-2\beta}p_{0}\right)^{'}}}\left(\int_{B_{R}} |h|^{p_{0}}dx\right)^{\frac{1-2\beta}{(1-\beta) p_{0} }}\right)^{1-\beta}\|h\|^{2\beta}
$$
 
$$
\leq ||u||^{q-1}  \left(\left(\int_{B_{R}}\left(|x|^{\frac{\alpha}{1-\beta}-\nu\frac{q-1}{1-\beta}}\right)^{\left(\frac{1-\beta}{1-2\beta} p_{0} \right)^{'}}dx\right)^{\frac{1}{\left(\frac{1-\beta}{1-2\beta}p_{0} \right)^{'}}}  C \,  \|h\|^{\frac{1-2\beta}{1-\beta}}\right)^{1-\beta}\|h\|^{2\beta}
$$

$$
=||u||^{q-1}  C \, \left(\int_{B_{R}}|x|^{\frac{\alpha-\nu(q-1)}{N-a_{0}+2(1-2\beta+a_{0}\beta)}2N}dx\right)^{\frac{N-a_{0}+2(1-2\beta+a_{0}\beta)}{2N}}\|h\|,
$$

\noindent where we used (\ref{eq:333-1}) and Lemma 2.3.

\item If $\beta=\frac{1}{2}$, there follows
$$
\frac{1}{\Lambda}\int_{B_{R}} K(|x|)|u|^{q-1}|h|dx\leq\int_{B_R}|x|^{\alpha}|u|^{q-1}V(|x|)^{\frac{1}{2}}|h|dx
$$

$$
\leq\left(\int_{B_{R}}|x|^{2\alpha}|u|^{2(q-1)}dx\right)^{\frac{1}{2}}\left(\int_{B_{R}} V(|x|)|h|^{2}dx\right)^{\frac{1}{2}}
$$

$$
\le ||u||^{q-1}  C \,  \left(\int_{B_{R}}|x|^{2\alpha-2\nu(q-1)}dx\right)\|h\|.
$$

\item If $\frac{1}{2}<\beta<1$, then we can apply H\"older's inequality first with conjugate exponents $2$, then $\frac{1}{2\beta-1}$, $\frac{1}{2(1-\beta)}$.
We get
$$
\frac{1}{\Lambda}\int_{B_{R}}K(|x|)|u|^{q-1}|h|dx
$$

$$
\leq\int_{B_{R}}|x|^{\alpha}V(|x|)^{\beta}|u|^{q-1}|h|dx\leq\int_{B_R}|x|^{\alpha}V(|x|)^{\frac{2\beta-1}{2}}|u|^{q-1}V(|x|)^{\frac{1}{2}} |h| dx
$$

$$
\le\left(\int_{B_{R}}|x|^{2\alpha}V(|x|)^{2\beta-1}|u|^{2(q-1)}dx\right)^{\frac{1}{2}}\left(\int_{B_{R}}V(|x|)|h|^{2}dx\right)^{\frac{1}{2}}
$$

$$
\leq\left(\int_{B_{R}}|x|^{2\alpha}|u|^{2(q-2\beta)}V(|x|)^{2\beta-1}|u|^{2(2\beta-1)}dx\right)^{\frac{1}{2}}\|h\|
$$

$$
\leq\left(\left(\int_{B_{R}}|x|^{\frac{\alpha}{1-\beta}}|u|^{\frac{q-2\beta}{1-\beta}}dx\right)^{2(1-\beta)}\left(\int_{B_{R}} V(|x|)|u|^{2}dx\right)^{2\beta-1}\right)^{\frac{1}{2}}\|h\|
$$

$$
\leq C \, ||u||^{q-2\beta} \, \left(\left(\int_{B_{R}}|x|^{\frac{\alpha}{1-\beta}-\nu\frac{q-2\beta}{1-\beta}}dx\right)^{2(1-\beta)}\left(\int_{B_{R}}V(|x|)|u|^{2}dx\right)^{2\beta-1}\right)^{\frac{1}{2}}\|h\|
$$

$$
=C \, ||u||^{q-2\beta} \, \left(\int_{B_{R}} |x|^{\frac{\alpha-\nu(q-2\beta)}{1-\beta}}dx\right)^{1-\beta}\left(\int_{B_{R}} V(|x|)|u|^{2}dx\right)^{\frac{2\beta-1}{2}}\|h\|
$$

$$
\leq C \, \left(\int_{B_{R}}|x|^{\frac{\alpha-\nu(q-2\beta)}{1-\beta}}dx\right)^{1-\beta}\|u\|_{A}^{q-1}\|h\| .
$$

\item  If $\beta=1$, then the hypothesis $q>\mathrm{max}\{1,2\beta\}$ implies $q>2$. Thus, we have
$$
\frac{1}{\Lambda}\int_{B_{R}} K(|x|)|u|^{q-1}|h|dx\leq\int_{B_{R}}|x|^{\alpha}V(|x|)|u|^{q-1}|h|dx
$$

$$
\leq\int_{B_{R}}|x|^{\alpha}V(|x|)^{\frac{1}{2}}|u|^{q-1}V(|x|)^{\frac{1}{2}}|h|dx
$$

$$
\leq\left(\int_{B_{R}}|x|^{2\alpha}V(|x|)|u|^{2(q-1)}dx\right)^{\frac{1}{2}}\left(\int_{B_{R}}V(|x|)|h|^{2}dx\right)^{\frac{1}{2}}
$$

$$
\leq\left(\int_{B_{R}}|x|^{2\alpha}|u|^{2(q-2)}V(|x|)|u|^{2}dx\right)^{\frac{1}{2}}\|h\|
$$

$$
\leq C  ||u||^{q-2}\left(\int_{B_{R}}|x|^{2\alpha-2\nu(q-2)}V(|x|)|u|^{2}dx\right)^{\frac{1}{2}}\|h\| .
$$

 \end{enumerate} The proof is now concluded.
\end{proof}

\bigskip

\begin{lem}
\label{A8} Let $N\geq3$ and $R>0$. Assume $[A]$, $[V],$ $[K]$. Assume also that

$$
\Lambda :=\mathsf{\mathnormal{\mathbf{\mathrm{ess\!\sup_{\mathit{x\in B_{R}^{c}}}}}}}\frac{K(|x|)}{|x|^{\alpha}V(|x|)^{\beta}}<+\infty\quad for\;some\;0\leq\beta\leq1\;and\;\alpha\in\mathbb{R}.
$$

\noindent Then $\forall u , h \in X$ and $\forall q>\mathrm{max}\{1,2\beta\}$ we have

$$
\int_{B_{R}^{c}}K(|x|)|u|^{q-1}|h|dx
$$

$$
\leq\begin{cases}
\Lambda ||u||^{q-1} \, C \,  s_{0}^{1-2\beta}\left(\int_{B_{R}^{c}}|x|^{\frac{\alpha-\nu(q-1)}{N-a_{\infty}+2(1-2\beta+a_{\infty}\beta)}2N}dx\right)^{\frac{N-a_{\infty}+2(1-2\beta+a_{\infty}\beta)}{2N}}\|h\| & 0\leq\beta\leq\frac{1}{2}\\
\Lambda  ||u||^{q-1}  \, C \,  \left(\int_{B_{R}^{c}}|x|^{\frac{\alpha-\nu(q-2\beta)}{1-\beta}}dx\right)^{1-\beta} \|h\| & \frac{1}{2}<\beta<1\\
\Lambda  ||u||^{q-2}  \, C \,  \left(\int_{B_{R}^{c}}|x|^{2\alpha-2\nu(q-2)} V(|x|)|u|^{2}dx\right)^{\frac{1}{2}}\|h\| & \beta=1
\end{cases}
$$
where $\nu:=\frac{N+a_{\infty}-2}{2}$ and $C=C (N,R, a_0, a_{\infty})$.\end{lem}

\noindent The proof of Lemma \ref{A8} is the same as that of Lemma \ref{A7}, and we will skip it.

\bigskip

\begin{defin}
For $\alpha\in\mathbb{R}$, $\beta\in[0,1]$ and $a\in(2-N,2]$ we define the functions $\alpha^{*}(a,\beta)$ and $q^{*}(a,\alpha,\beta)$ as follows:
$$
\alpha^{*}(a,\beta):=\max \left\{2\beta-1-\frac{N}{2}-a\beta+\frac{a}{2},-(1-\beta)N \right\}
$$

$$
=\begin{cases}
2\beta-1-\frac{N}{2}-a\beta+\frac{a}{2} & if \;0\leq\beta\leq\frac{1}{2}\\
-(1-\beta)N & if \;\frac{1}{2}\leq\beta\leq1
\end{cases}
$$

$$
q^{*}(a,\alpha,\beta):=2\, \frac{\alpha-2\beta+N+a\beta}{N+a-2}.
$$

\end{defin}

\bigskip

\medskip{}

\begin{thm}
\label{theor0}

Let $N\geq3$.  Assume $[A]$, $[V],$ $[K]$. Assume also that there exists $R_{1}>0$ such that
\begin{equation}
\mathsf{\mathnormal{\mathbf{\mathrm{ess\!\!\!\!\sup_{\mathit{x \in B_{R_1}}}}}}}\frac{K(|x|)}{|x|^{\alpha_{0}}V(|x|)^{\beta_{0}}}<+\infty\quad for\;some\;0\leq\beta_{0}\leq1\;and\;\alpha_{0}>\alpha^{*}(a_{0},\beta_{0}).\label{eq:228}
\end{equation}
Then $\lim_{R\rightarrow0^{+}}\mathcal{S}_{0}(q_{1},R)=0$ for any $q_{1}\in\mathbb{R}$ such that
\begin{equation}
\mathrm{max}\{1,2\beta_{0}\}<q_{1}<q^{*}(a_{0},\alpha_{0},\beta_{0}).\label{eq:12-1}
\end{equation}
\end{thm}
\begin{proof}
Let $u, h\in X$ satisfy $\|u\|=\|h\|=1$. Take $R$ such that $0<R <R_1$. Of course we have
$$
\mathsf{\mathnormal{\mathbf{\mathrm{ess\!\sup_{\mathit{x\in B_{R}} }}}}}\frac{K(|x|)}{|x|^{\alpha_{0}}V(|x|)^{\beta_{0}}}\leq\mathsf{\mathnormal{\mathbf{\mathrm{ess\!\!\!\!\sup_{\mathit{x\in B_{{R_1}}} }}}}}\frac{K(|x|)}{|x|^{\alpha_{0}}V(|x|)^{\beta_{0}}}<+\infty,
$$

\noindent so we can apply Lemma \ref{A7} with $\alpha=\alpha_{0}$ and $\beta=\beta_{0}$ and $h=u$. 

\begin{enumerate}

\item 

If $0\leq\beta_{0}\leq\frac{1}{2}$ we obtain 

$$
\int_{B_{R}}K(|x|)|u|^{q_{1}}dx= \int_{B_{R}}K(|x|)|u|^{q_{1}-1} \, |u| \,dx
$$

$$
\leq C\left(R^{\frac{2\alpha_{0}-4\beta_{0}+2N+2a_{0}\beta_{0}-(N+a_{0}-2)q_{1}}{N-a_{0}+2(1-2\beta_{0}+a_{0}\beta_{0})}N}\right)^{\frac{N-a_{0}+2(1-2\beta_{0}+a_{0}\beta_{0})}{2N}},
$$

\noindent since

$$
\frac{\alpha_{0}-\nu(q-1)}{N-a_{0}+2(1-2\beta_{0}+a_{0}\beta_{0})}2N +N ={\frac{2\alpha_{0}-4\beta_{0}+2N+2a_{0}\beta_{0}-(N+a_{0}-2)q_{1}}{N-a_{0}+2(1-2\beta_{0}+a_{0}\beta_{0})}N}
$$

\noindent and it is easy to check that

$$
2\alpha_{0}-4\beta_{0}+2N+a_{0}\beta_{0}-(N+a_{0}-2)q_{1}=(N+a_{0}-2)(q^{*}(a_{0},\alpha_{0},\beta_{0})-q_{1})>0,
$$

$$
N-a_{0}+2(1-2\beta_{0}+a_{0}\beta_{0})\geq N>0.
$$

\item If $\frac{1}{2}<\beta_{0}<1$ we get 
$$
\int_{B_{R}}K(|x|)|u|^{q_{1}}dx  =  \int_{B_{R}}K(|x|)|u|^{q_{1}-1} \, |u| \, dx  \leq C\left(R^{\frac{2\alpha_{0}-(N+a_{0}-2)(q_{1}-2\beta_{0})}{2(1-\beta_{0})}+N}\right)^{1-\beta_{0}},
$$

\noindent because

$$
\frac{\alpha_{0}-\nu(q_1 -2\beta_{0})}{1-\beta_{0}}+N=\frac{\alpha_{0}- \frac{N+a_{0}-2}{2}  (q_1 -2\beta_{0})}{1-\beta_{0}}
+N
$$

$$
=\frac{2\alpha_{0}-(N+a_{0}-2)(q_{1}-2\beta_{0})}{2(1-\beta_{0})}+N=\frac{N+a_{0}-2}{2(1-\beta_{0})}(q^{*}(a_{0},\alpha_{0},\beta_{0})-q_{1})>0.
$$

\item Finally, if $\beta_{0}=1$, it holds that

$$
\int_{B_{R}}K(|x|)|u|^{q_{1}}dx  = \int_{B_{R}}K(|x|)|u|^{q_{1}-1} \, |u| \, dx   \leq CR^{\frac{2\alpha_{0}-(N+a_{0}-2)(q_{1}-2)}{2}},
$$

\noindent because

$$
2\alpha_{0}-2\nu(q_1-2)=  2\alpha_{0}-2   \frac{N+a_{0}-2}{2} (q_1-2)
$$

$$
=2\alpha_{0}-(N+a_{0}-2)(q_{1}-2)=(N+a_{0}-2)(q^{*}(a_{0},\alpha_{0},1)-q_{1})>0.
$$

\end{enumerate}

\noindent Hence, in any of the previous cases there exist a constant $\delta=\delta(N,a_{0},\alpha_{0},\beta_{0},q_{1})>0$ such that 
$$
\mathcal{R}_{0}(q_{1},R)\leq CR^{\delta} \rightarrow 0 \quad \mbox{as} \, R \rightarrow 0,
$$

\noindent from which our thesis follows.
\end{proof}

\begin{thm}

Let $N\geq3$. Assume $[A]$, $[V],$ $[K]$. Assume also that there exists $R_{2}>0$ such that

$$
\mathsf{\mathnormal{\mathbf{\mathrm{ess\!\sup_{\mathit{|x|>R_{2}}}}}}}\frac{K(|x|)}{|x|^{\alpha_{\infty}}V(|x|)^{\beta_{\infty}}}<+\infty\quad for\;some\;0\leq\beta_{\infty}\leq1\;and\;\alpha_{\infty}\in\mathbb{R}.
$$

\noindent Then $\lim_{R\rightarrow+\infty}\mathcal{S}_{\infty}(q_{2},R)=0$ for each $q_{2}\in\mathbb{R}$ such that
\begin{equation}
q_{2}>\mathrm{max}\{1,2\beta_{\infty},q^{*}(a_{\infty},\alpha_{\infty},\beta_{\infty})\}.\label{eq:24-1}
\end{equation}
\end{thm}

\begin{proof}
Let $u \in X$ satisfy $\|u\|=1$. Consider $R\geq R_2$, of course we have 
$$
\mathsf{\mathnormal{\mathbf{\mathrm{ess\!\sup_{\mathit{x\in B_{R}^{c}}}}}}}\frac{K(|x|)}{|x|^{\alpha_{\infty}}V(|x|)^{\beta_{\infty}}}\leq\mathnormal{\mathbf{\mathrm{ess\!\sup_{\mathit{|x|>R_{2}}}}}}\frac{K(|x|)}{|x|^{\alpha_{\infty}}V(|x|)^{\beta_{\infty}}}<+\infty,
$$

\noindent hence we can apply Lemma \ref{A8} with $\alpha=\alpha_{\infty}$ and $\beta=\beta_{\infty}$. The arguments are the same as in Theorem \ref{theor0}, so we will skip the details.

\begin{enumerate}
\item

If $0\leq\beta_{\infty}\leq\frac{1}{2}$, with similar considerations of those used for $\beta_{0}$, we find
$$
\int_{B_{R}}K(|x|)|u|^{q_{2}}dx =\int_{B_{R}}K(|x|)|u|^{q_{2}-1} \, |u| \,dx
$$

$$
\leq C\left( 
R^{\frac{2\alpha_{\infty}-4\beta_{\infty}+2N+2a_{\infty}\beta_{\infty}-(N+a_{\infty}-2)q_{2}}{N-a_{\infty}+2(1-2\beta_{\infty}+a_{\infty}\beta_{\infty})}N }
\right)^\frac{N-a_{\infty}+2(1-2\beta_{\infty} + a_{\infty} \beta_{\infty})}{2N} $$

\noindent since

$$
2\alpha_{\infty}-4\beta_{\infty}+2N+a_{\infty}\beta_{\infty}-(N+a_{\infty}-2)q_{2}<0,
$$

$$
N-a_{\infty}+2(1-2\beta_{\infty}+a_{\infty}\beta_{\infty})\geq N>0.
$$

\item On the other hand, if $\frac{1}{2}<\beta_{\infty}<1$, we have 
$$
\int_{B_{R}}K(|x|)|u|^{q_2}dx = \int_{B_{R}}K(|x|)|u|^{q_2 -1} \, |u| \, dx   \leq C\left(R^{\frac{2\alpha_{\infty}-(N+a_{\infty}-2)(q_{2}-2\beta_{\infty})}{2(1-\beta_{\infty})}+N}\right)^{1-\beta_{\infty}},
$$

\noindent as 

$$
\frac{2\alpha_{\infty}-(N+a_{\infty}-2)(q_{2}-2\beta_{\infty})}{2(1-\beta_{\infty})}+N<0.
$$

\item Finally, if $\beta_{\infty}=1$, we obtain
$$
\int_{B_{R}}K(|x|)|u|^{q_{2}}dx= \int_{B_{R}}K(|x|)|u|^{q_2 -1} \, |u| \, dx \leq CR^{\frac{2\alpha_{\infty}-(N+a_{\infty}-2)(q_{2}-2)}{2}},
$$

\noindent because

$$
2\alpha_{\infty}-(N+a_{\infty}-2)(q_{2}-2)<0.
$$

\end{enumerate} 
In each of the previous cases, there exists $\delta=\delta(N,a,\alpha_{\infty},\beta_{\infty},q_{2})>0$ such that
$$
\mathcal{S}_{\infty}(q_{2},R)\leq CR^{-\delta} \rightarrow 0, \quad \mbox{as} \, R\rightarrow \infty ,
$$

\noindent from which our thesis follows. 
\end{proof}

\bigskip

\noindent From the previous theorems we easily derives our main compactness result.

\bigskip

\begin{thm} \label{thmcomp}
 Assume $N\geq3$.  Assume $[A]$, $[V],$ $[K]$. Moreover, assume the hypotheses of the two previous theorems, that is: there are $R_1 , R_2 >0$, $\alpha_0 , \alpha_{\infty} \in \mathbb{R}$, $\beta_0 , \beta_{\infty} \in [0,1]$ such that
$$
\mathsf{\mathnormal{\mathbf{\mathrm{ess\sup_{|x|<R_1}}}}}\frac{K(|x|)}{|x|^{\alpha_{0}}V(|x|)^{\beta_{0}}}<+\infty,\;\mathbf{\mathrm{ess\sup_{|x|\geq R_2}}}\frac{K(|x|)}{|x|^{\alpha_{\infty}}V(|x|)^{\beta_{\infty}}}<+\infty .
$$

\noindent Thus, for $q_{1}$ and $q_{2}$ such that
$$
\begin{cases}
q_{1}\in\mathcal{I}_{1}=(\mathrm{max}\{1,2\beta_{0}\},\,q^{*}(a_{0},\alpha_{0},\beta_{0}))\\
q_{2}\in\mathcal{I}_{2}=(\mathrm{max}\{1,2\beta_{\infty},\,q^{*}(a_{\infty},\alpha_{\infty},\beta_{\infty})\},+\infty)
\end{cases}
$$

\noindent the embedding
$$
X\hookrightarrow L_{K}^{q_{1}}(\mathbb{R}^{N})+L_{K}^{q_{2}}(\mathbb{R}^{N}).
$$

\noindent is continuous and compact.
\end{thm}

\bigskip
\bigskip

\section{Applications: existence results}

We now use these results on compact embeddings to obtain existence and multiplicity results for nonlinear elliptic equations. 
We will deal with equation (\ref{EQ}), and we will assume hypotheses $[A],[V],[K]$. As to the nonlinearity $f$ we will assume the following hypotheses.

\begin{itemize}

\item[{$(f_1 )$}] $f: \mathbb{R} \rightarrow \mathbb{R} $ is a continuous functions, and there are constants $q_1 , q_2 >2$ and 
$M>0$ such that 
$$|f(t)| \leq M \min\{ |t|^{q_1 -1}, |t|^{q_2 -1} \}, \quad \mbox{for all} \,\, t \in  \mathbb{R} $$

\item[{$(f_2 )$}] Define $F(t) = \int_0^{t} f(s) \, ds$, then there is $\theta >2$ such that $0 \leq \theta F(t) \leq f(t)t $ for all $t$. Furthermore there is $t_0 >0$ such that $F(t_0 )>0$.

\end{itemize}

\noindent The simplest example of a function satisfying $(f_1 )$, $(f_2 )$ is given by
$$f(t)= \min \{ t^{q_1 -1},  t^{q_2 -1}  \}$$

\noindent  if $t \geq 0$, and $f(t)=-f(-t)$ if $t \leq 0$ (or also $f(t)=0$ if $t \leq 0$), with $q_1 , q_2 >2$. Notice that if $q_1 \not= q_2$ there is no pure power function, i.e. $f(t)=t^q$, satisfying $(f_1 )$. However we do not assume $q_1 \not= q_2$, so pure power functions are included in our results, when the hypotheses will allow to choose $q_1 =q_2$.

\noindent We define the functional $I : X \rightarrow  \mathbb{R}$

$$I(u)= \frac{1}{2} \int_{\mathbb{R}^N} A(|x|) \, |\nabla u |^2 \, dx + \frac{1}{2} \int_{\mathbb{R}^N} V(|x|) \, u^2 \, dx -
\int_{\mathbb{R}^N} K(|x|) \, F(u) \, dx .$$

\bigskip

\begin{thm}
\label{theorboh1}
Assume the hypotheses of Theorem \ref{thmcomp}. Assume $(f_1 ), (f_2 )$ with $q_i \in \mathcal{I}_{i}$, where the intervals $\mathcal{I}_{i}$ are given in Theorem \ref{thmcomp}. Then $I$ is a $C^1$ functional on $X$, whose differential is given by

$$I'(u)h =\int_{\mathbb{R}^N} A(|x|) \, \nabla u \nabla h \, dx + \int_{\mathbb{R}^N} V(|x|) u\, h \, dx - \int_{\mathbb{R}^N} K(|x|) \, f(u) h \, dx $$

\noindent for all $u, h \in X .$
\end{thm}
\begin{proof}
We know that the embedding of $X$ in $L_{K}^{q_1}+L_{K}^{q_2}$ is continuous. By the previous results and proposition 3.8 of \cite{BPR} we 
also know that the functional

$$\Phi(u) =  \int_{\mathbb{R}^N} K(|x|) \, F(u) dx$$

\noindent is of class $C^1$ on $L_{K}^{q_1}+L_{K}^{q_2}$, with differential given by

$$\Phi' (u)h =  \int_{\mathbb{R}^N} K(|x|) \, f(u)h dx .$$

\noindent Obviously the quadratic part of $I$ is $C^1$, with differential given by

$$h \rightarrow  \int_{\mathbb{R}^N}A(|x|) \nabla u \nabla h dx + \int_{\mathbb{R}^N}V(|x|) f(u) h   dx .$$

\noindent The thesis easily follows.

\end{proof}

\bigskip

\begin{thm}
\label{theorPS}
Assume the hypotheses of Theorem \ref{thmcomp}. Assume $(f_1 ), (f_2 )$ with $q_i \in \mathcal{I}_{i}$, where the intervals $\mathcal{I}_{i}$ are given in Theorem \ref{thmcomp}. Then $I : X \rightarrow \mathbb{R}$ satisfies the Palais-Smale condition.

\end{thm}
\begin{proof}

\noindent Assume that $\{ u_n \}_n$ is a sequence in $X$ such that $I(u_n )$ is bounded and $I'(u_n ) \rightarrow 0$ in $X'$. We have to prove that $\{ u_n \}_n$ has a converging subsequence. For this, notice that from the hypotheses we derive, for a suitable positive constant $C$,

$$C + C ||u_n || \geq I(u_n ) + \frac{1}{\theta} I'(u_n ) u_n = \left( \frac{1}{2} - \frac{1}{\theta} \right) \, ||u_n ||^2 + \int_{\mathbb{R}^N} \left( \frac{1}{\theta} f(u_n ) u_n - F(u_n ) \right) \, dx $$
$$\geq \left( \frac{1}{2} - \frac{1}{\theta} \right) \, ||u_n ||^2 , $$

\noindent and this implies that $\{ u_n \}_n$ is bounded. So we can assume, up to a subsequence, $u_n \rightharpoonup u$ in $X$ and $u_n \rightarrow u$ in $L_{K}^{q_{1}}(\mathbb{R}^{N})+L_{K}^{q_{2}}(\mathbb{R}^{N})$. Now we have

$$||u_n - u||^2 = (u_n | u_n -u ) - (u | u_n -u) = I'(u_n ) (u_n -u) + \Phi '(u_n ) (u_n - u) - (u | u_n -u). $$

\noindent Of course $(u | u_n -u)\rightarrow 0$ because  $u_n \rightharpoonup u$ in $X$. We also have that $I' (u_n ) \rightarrow 0$ in $X'$ while $u_n -u$ is bounded in $X$, so $I'(u_n ) (u_n -u) \rightarrow 0$. Lastly, we know that $\Phi$ is $C^1$ in the space $L_{K}^{q_{1}}(\mathbb{R}^{N})+L_{K}^{q_{2}}(\mathbb{R}^{N})$, and $u_n \rightarrow u$ in that space, so $\Phi ' (u_n )$ is bounded (as a sequence in the dual space) and $u_n -u \rightarrow 0$, so $\Phi '(u_n ) (u_n - u)\rightarrow 0$. Hence we get $||u_n - u||^2 \rightarrow 0$, which is the thesis.

\end{proof}

\begin{thm}
\label{theorboh3}
Assume the hypotheses of Theorem \ref{thmcomp}. Assume $(f_1 ), (f_2 )$ with $q_i \in \mathcal{I}_{i}$, where the intervals $\mathcal{I}_{i}$ are given in Theorem \ref{thmcomp}. Then $I : X \rightarrow \mathbb{R}$ has a non negative and non trivial critical point.

\end{thm}
\begin{proof}

Firstly, to have non negative solution, we assume as usual $f(t)=0 $ for $t\leq 0$. To prove the theorem we apply the standard Mountain Pass Lemma. We have proven that $I$ satisfies the Palais-Smale condition, so it is enough to prove that it has the usual Mountain Pass geometry, that is, we have to prove the following two conditions:
\begin{itemize}

\item[{$(i )$}] There are $\rho , \alpha >0$ such that $I(u) \geq \alpha$ for all $||u||= \rho$.

\item[{$(ii )$}]There is $v \in X$ such that $||v|| \geq \rho$ and $I(u) \leq 0$.

\end{itemize}

\noindent As for $(i)$, let us take $0<R_1 <R_2 $ such that

$$\mathcal{S}_0 (q_1 ,R_1 ) < +\infty , \quad \mathcal{S}_{\infty} (q_2 , R_2) < +\infty ,$$

\noindent which is possible because $q_i \in \mathcal{I}_{i}$. Then, using the definition of $\mathcal{S}_0, \mathcal{S}_{\infty}$, lemma \ref{A5}  and the embedding of $X$ in $L^2 (B_{R_2} \backslash B_{R_1} )$, we get

$$\int_{ B_{R_1} } K(|x|) |u|^{q_1} \leq  c_1 ||u||^{q_1} , \int_{B_{R_2}^c } K(|x|) |u|^{q_2} \leq c_2 ||u||^{q_2}  , \int_{B_{R_2} \backslash B_{R_1} } K(|x|) |u|^{q_1} \leq c_1 ||u||^{q_i}  .$$

\noindent Hence

$$ \Big|  \int_{\mathbb{R}^N} K(|x|) F(u)\, dx \Big| \leq $$
$$ \int_{B_{R_1}}  K(|x|) F(u)\, dx +  \int_{B_{R_2} \backslash B_{R_1} }  K(|x|) F(u)\, dx +  \int_{{\mathbb{R}^N} \backslash B_{R_2}}  K(|x|) F(u)\, dx \leq$$
$$M  \int_{B_{R_1}}  K(|x|) |u|^{q_1}\, dx + c_1 ||u||^{q_1} + M  \int_{{\mathbb{R}^N} \backslash B_{R_2}}   K(|x|) |u|^{q_2}\, dx
\leq$$
$$  c_3 ||u||^{q_1} + c_4 ||u||^{q_2}, $$ 

\noindent so that

$$I(u) \geq \frac{1}{2} ||u||^2 -  c_3 ||u||^{q_1} - c_4 ||u||^{q_2}, $$ 

\noindent and $(i)$ easily follows. 
\par \noindent To get $(ii)$, we start remarking that, from $(f_2 )$, there is $c>0$ such that, for all $t\geq t_0$, $F(t) \geq c \, t^{\theta}$. The potential $K$ is not zero a.e., and from this fact it is easy to deduce that there are $\delta >0$ and a measurable subset $A_{\delta} \subset (\delta , 1/\delta )$ such that $|A_{\delta}| > \delta $ and $K(r) > \delta $ in $A_{\delta}$. Now take a function $\varphi \in 
C^{\infty} (\mathbb{R})$ such that $0\leq \varphi (r)\leq 1$ for all $r$, $\varphi (r)=1 $ for $r \in (\delta , 1/\delta )$, $\varphi (r)=0 $ for $r \leq 
\frac{1}{2\delta}$ and $r \geq 1+1/\delta$. Define now, for $x \in \mathbb{R}$, $\psi (x) = \varphi (|x|)$. As $\varphi \in C_{c}^{\infty}(\mathbb{R}\backslash \{ 0 \})$, it is $ \psi \in C_{c}^{\infty}(\mathbb{R}^N \backslash \{ 0 \})$, and furthermore $\psi $ is radial, so $\psi \in X$. Define $\Omega_{\delta}= \{ x \in  \mathbb{R}^N \, | \, |x| \in A_{\delta} \}$. Hence, if we take $\lambda > t_0$ we get 

$$K(|x|) F( \lambda \psi (x) ) \geq \delta \, c \lambda^{\theta} \quad \mbox{in} \,\, \Omega_{\delta},$$

\noindent and $K(|x|) F( \lambda \psi (x) )\geq 0$ for all $x$, so that 

$$\int_{\mathbb{R}^N}K(|x|) F( \lambda \psi (x) ) \geq\int_{\Omega_{\delta}}K(|x|) F( \lambda \psi (x) )\geq c \delta \, \lambda^{\theta} |\Omega_{\delta}|= C_{\delta} \lambda^{\theta} ,$$

\noindent where $C_{\delta} >0$ depends only on $\delta $ and $N$. We then get

$$I(\lambda \psi)\leq  \lambda^2 ||\psi||^2 -C_{\delta} \lambda^{\theta},$$

\noindent so $I(\lambda \psi) \rightarrow -\infty$ as $\lambda \rightarrow + \infty$, and this gives the result.

\end{proof}

\bigskip

\noindent As $I$ satisfies the Palais-Smale condition, arguing as in the proof of Theorem 1.2 in \cite{BPR}, we also get a result of existence of infinity solutions.

\bigskip

\begin{thm}
\label{theorboh4}
Assume the hypotheses of Theorem \ref{thmcomp}. Assume $(f_1 ), (f_2 )$ with $q_i \in \mathcal{I}_{i}$, where the intervals $\mathcal{I}_{i}$ are given in Theorem \ref{thmcomp}.  Assume furthermore the following two assumptions

\begin{itemize}

\item[{$(f_3 )$}] There exists $m>0$ such that $F(t) \geq m \, \min \{ t^{q_1 },t^{q_2 }   \}$ for all $t>0$.

\item[{$(f_4)$}] $f$ is an odd function.

\end{itemize}

Then $I : X \rightarrow \mathbb{R}$ has a sequence $\{  u_n \}_n$ of critical points such that $I(u_n ) \rightarrow + \infty$.

\end{thm}

\bigskip
\bigskip

\section{Examples}

\noindent In this section we give some examples that could help to understand what is new (and what is not) in our results. We will make a comparison, in some concrete cases, between our results and those of \cite{Su-Wang-Will-q}. In this paper the authors study a p-laplacian equation, so we compare their results with ours only in the case $p=2$. Our problem is also linked to those studied in \cite{ChenChen-X} and \cite{Huang-X}, but in the following examples we assume $A(r)= \min\{r^{\alpha}, r^{\beta}  \}$, with $\alpha \not= \beta$, and this rules out the results of \cite{ChenChen-X} and \cite{Huang-X}, in which it is $A(r)=r^{\alpha}$, for some $\alpha \in \mathbb{R}$. In \cite{Su-Wang-Will-q}, the authors define three functions $q^* , q_{*}, q_{**}$ which depends on the asymptotic behavior of the potentials, and find existence of solutions for, say, $f(t)= t^{q_1 -1}+t^{q_2 -1}$ when $q_i \in ( q_{*} , q^* )$ or when $q_i > \max \{q_{*}, q_{**} \}$ $(i=1,2)$.

\bigskip

{\bf 1.}  Let us choose the functions $A, V, K$ as follows:

$$A(r)= \min \{ r^2 , r^{3/2} \}; \quad V(r)= \min \left\{ 1, \frac{1}{r^{1/2}}  \right\}; \quad K(r)= \max \{ r^{1/2}, r^{3/2}  \}.$$

\noindent It is simple to verify that in this case the results of \cite{Su-Wang-Will-q} do not apply, because if we compute the functions $q^*$ and $q_*$ it happens $ q_*  = \frac{4N +6}{2N-1}$ and $q^* = \frac{2N+1}{N}$ (and $q_{**}$ is not defined), so that $q^* < q_*$ while $q^* > q_*$ is a needed hypothesis. To apply our results, we can choose $\beta_0 =\beta_{\infty}=0$, $\alpha_0 = 1/2$, 
$\alpha_{\infty}= 3/2$, $a_0= 2$, $a_{\infty} =3/2$. We then get $q^* (a_0, \alpha_0 , \beta_0 ) = \frac{2N+1}{N}>2$ and 
$q^* (a_{\infty}, \alpha_{\infty} , \beta_{\infty} ) = \frac{4N +6}{2N-1}$. Hence, if we choose

$$2<q_1 <\frac{2N+1}{N} <\frac{4N +6}{2N-1} <q_2 ,$$

\noindent and $f(t)= \min \{ t^{q_1 -1 },  t^{q_2 -1 }   \}$, we can apply our existence results. Notice that in this case $\mathcal{I}_{1} \cap \mathcal{I}_{2} = \emptyset$.

\bigskip

{\bf 2.}  Assume $N\geq 6$ and choose the functions $A, V, K$ as follows:

$$A(r)= \max \left\{ \frac{1}{r^2}, \frac{1}{r^3} \right\} ; \quad V(r)= \frac{1}{r^{4}} ; \quad K(r)= \min \left\{ 1, \frac{1}{r^{2}}   \right\}.$$

\noindent In this case, following \cite{Su-Wang-Will-q}, the computations give $q^* = \frac{2N}{N-5}$ and $q_* =\frac{2(N-2)}{N-4}$, and it is $q_* < q^*$ for $N\geq 3$, so in this case they get existence for $q_i \in \left(\frac{2(N-2)}{N-4}, \frac{2N}{N-5}\right)$. To apply our results, we can choose $\beta_0 =\beta_{\infty}= \alpha_0 = 0$, 
$\alpha_{\infty}= -2$, $a_0= -3$ and $a_{\infty} =-2$. We then get $q^* (a_0, \alpha_0 , \beta_0 ) = \frac{2N}{N-5}$ and 
$q^* (a_{\infty}, \alpha_{\infty} , \beta_{\infty} ) =\frac{2(N-2)}{N-4}$, so ${\cal{I}}_1 \cap {\cal{I}}_2 = \left(\frac{2(N-2)}{N-4}, \frac{2N}{N-5}\right)$, the same interval. Hence, for pure power functions we obtain exactly the same result as in \cite{Su-Wang-Will-q}, while we can not treat functions like $f(t)= t^{q_1 -1}+t^{q_2 -1}$. On the other hand, we are free to choose $2<q_1 < \frac{2(N-2)}{N-4}<q_2$ and $f(t)= \min \{ t^{q_1 -1}, t^{q_2 -1}  \}$ and such a function does not satisfy the hypotheses of \cite{Su-Wang-Will-q}, because it satisfies $(f_2)$ with $\theta = q_1 < \frac{2(N-2)}{N-4}$, which is not allowed in \cite{Su-Wang-Will-q}.

\bigskip

{\bf 3.}  Finally, assume again $N\geq 6$ and choose the functions $A, V, K$ as follows:

$$A(r)=  \max \left\{ \frac{1}{r^2}, \frac{1}{r^3} \right\} ; \quad V(r)= e^{2r} ; \quad K(r)= e^r  .$$

\noindent In this case the results of \cite{Su-Wang-Will-q} do not apply because of the exponential growth of the potential $K$. We can choose $a_0= -3 $, $a_{\infty} =-2$, 
$\beta_0 =\alpha_{0}= \alpha_{\infty}=0$, $\beta_{\infty}=\frac{1}{2}$, and we get, as before, $q^* (a_0, \alpha_0 , \beta_0 ) = \frac{2N}{N-5}$ and $q^* (a_{\infty}, \alpha_{\infty} , \beta_{\infty} ) =\frac{2(N-2)}{N-4}$, so again we get existence of solution for functions like $f(t)= \min \{ t^{q_1 -1}, t^{q_2 -1}  \}$ and for the same range of exponents $q_i$ as above. In particular we can choose $f(t)= t^{q-1}$ for $q \in \left(\frac{2(N-2)}{N-4}, \frac{2N}{N-5}\right)$.

\end{document}